\documentclass{amsart}
\usepackage[a4paper,hcentering]{geometry}
\usepackage{times}
\usepackage{amsthm}
\usepackage{amsmath}
\usepackage{amssymb}
\usepackage{amsfonts}
\usepackage{enumerate}

\def\eg,{\textit{e.g.,}}
\def\etc.{\textit{etc.}}
\def\ie,{\textit{i.e.,}}
\def\ax#1{A\ref{ax#1}}
\def\axioms/{A1--A4}
\def\lem#1{L\ref{lem#1}}

\theoremstyle{definition}
\newtheorem*{definition}{Definition}
\newtheorem*{remark}{Remark}
\newtheorem*{example}{Example}
\theoremstyle{plain}
\newtheorem{lemma}{Lemma}
\newtheorem{theorem}[lemma]{Theorem}
\newtheorem*{claim}{Claim}

\def\-#1{{\overline{#1}}}
\let\~\widetilde
\def\S#1{{#1^{\{2\}}}}
\def\pmod#1{\ (\mathrm{mod}\:#1)}
\def\rank{\mathop{\mathrm{rank}}\nolimits}
\def\ind{\mathop{\mathrm{ind}}\nolimits}
\def\GL{\mathop{\mathit{GL}}\nolimits}
\def\K#1{{#1}/{\pm1}}
\def\tK#1#2{{#1}/{#2}}
\def\omicron{{\rm o}}
\let\epsilon\varepsilon
\let\oldstar\star
\def\star{\mathord{\oldstar}}

\def\NN{{\mathbb{N}}}
\def\PP{{\mathbb{P}}}
\def\RR{{\mathbb{R}}}
\def\ZZ{{\mathbb{Z}}}
\def\diagcross{\text{\rlap{$\diagup$}$\diagdown$}}
\def\diagdots{\text{\raisebox{3pt}{$\cdot$}$\cdot$\raisebox{-3pt}{$\cdot$}}}%
\let\implies\Rightarrow

\let\iff\Leftrightarrow
\let\tofrom\leftrightarrow

\title{Kummer structures}
\author{Adam Chalcraft\\Michael Fryers}
\begin{document}
\maketitle

Suppose we take an abelian group $G$ and let $\K{G}$ be the quotient of $G$
by the action of negation.
What structure does $\K{G}$ inherit from the group structure of $G$?

Let us write $\-x\in \K{G}$ for the image of $x\in G$.
Then given $\-x,\-y\in \K{G}$, we cannot define $\-{x+y}$ uniquely,
because $\-{-x}=\-x$ but $\-{-x+y}\neq\-{x+y}$ in general;
but we can define the unordered pair $\{\-{x+y},\-{x-y}\}$.
We thus get a map $\kappa$ from $\S{(\K{G})}$, the set of unordered pairs of
elements of $\K{G}$, to itself.
We call the structure $(\K{G},\kappa)$ the \emph{Kummer} of $G$.

An example from geometry explains our use of the name \emph{Kummer}:
the quotient of an Abelian surface (\ie, a two-dimensional
projective algebraic group) by $\pm1$ is called a \emph{Kummer surface}.

In this paper we propose some axioms that hold for the structure
$(\K{G},\kappa)$, and show that every structure satisfying those axioms
either is the Kummer of a unique group, or comes from one other construction,
the quotient of a 2-torsion group by an involution.
The proofs are constructive, showing how $G$ can be reconstructed from
$(\K{G},\kappa)$.

\section{Axioms}\label{secAxioms}

\begin{definition}
  A \emph{Kummer structure} is a set $K$ with a map $\kappa\colon\S K\to\S K$
  satisfying the following axioms:
  (we use the notation $a\ b\to c\ d$ to mean
  $\{a,b\}\stackrel\kappa\longmapsto\{c,d\}$)

  \begin{enumerate}[{A}1.]
  \def\itemax#1{\item\label{ax#1}}%
  \itemax{id} There is an element $0\in K$ such that for every $a\in K$
    we have $a\ 0\to a\ a$.
  \itemax{aa02a} For every $a\in K$ there is an element $2a\in K$ such that
    $a\ a\to 0\ 2a$.
  \itemax{assoc} Every $a,b,c\in K$ fit into a diagram like this:
    $$
    \begin{array}{*5c@{}l}
      &c&&c&c\\
      a&b&\to&p_0&p_1\\
      &\downarrow&&\downarrow&\downarrow\\
      a&q_0&\to&s_0&s_1\\
      a&q_1&\to&s_2&s_3\\
    \end{array}
    $$
    (here the downward arrows mean $c\ b\to q_0\ q_1$ \etc.).
  \itemax{2ishom} $2$ is a homomorphism: that is, if $a\ b\to c\ d$ then $2a\ 2b\to 2c\ 2d$.
  \end{enumerate}
\end{definition}

\begin{remark}
  It is clear that the element $0$ and the map $2$ described in the axioms are
  uniquely determined by $\kappa$.
  
  We will write $a\ b\to c\ \star$ to mean that $a\ b\to c\ d$ for some
  unspecified $d$, and use the obvious notation $4a:=2(2a)$.
\end{remark}

\section{Examples}\label{secExamples}

\begin{definition}
  Given an abelian group $(G,+)$,
  the \emph{Kummer} of $G$ is $(\K{G},\kappa)$, where
  $$
    \K{G}:=\{\,\{x,-x\}\mid x\in G\,\},
  $$
  and (writing $\-x$ for $\{x,-x\}$) $\kappa$ is defined by
  $$
  \begin{array}{*5c@{}l}
    \-x&\-y&\to&\-{x+y}&\-{x-y}&.\\
  \end{array}
  $$
\end{definition}

\begin{remark}
  It is easy to check that $\K{G}$ satisfies axioms \axioms/ with
  $0=\-0$ and $2\-x=\-{2x}$.
\end{remark}

\begin{example}
  Let $K$ be the closed interval $[-1,1]$, and let $\kappa$ be defined by
  the rule
  $$
    a\ b\to c\ \star\quad\iff\quad a^2+b^2+c^2 = 2abc+1
  $$
  (so $\kappa\{a,b\}$ is the set of solutions of this equation.)
  Then the axioms \axioms/ can be checked by hand.  Alternatively, one
  may observe that this $K$ is the Kummer of the group $\RR/2\pi\ZZ$, where
  $\-\theta$ is represented by $\cos\theta\in [-1,1]$.
\end{example}

\begin{example}
  (For algebraic geometers)
  If $G=E(k)$ is the set of points of an elliptic curve in Weierstra\ss\ form
  over an algebraically closed field $k$, then $K=\K{G}$ can be identified
  with $\PP^1(k)$ by representing a point $\-{(x,y)}$ by $x$.
  The map $\kappa$ can be given by an algebraic formula like that in the
  previous example.  If $k$ is not algebraically closed, then $K$ is the
  subset of $\PP^1(k)$ consisting of $x$-coordinates of $k$-points of $E$.
\end{example}

Since the axioms were intended to capture the structure of Kummers
of abelian groups, it would not be surprising if this were the only example
of a Kummer structure.
There is, however, another construction, which we will later show leads
to Kummer structures not isomorphic to the Kummers of groups:

\begin{definition}
  Let $G$ be a 2-torsion abelian group, and $\iota\colon G\to G$ an involution
  (that is, an automorphism with $\iota^2=1$).
  The \emph{twisted Kummer} of $(G,\iota)$ is $(\tK{G}{\iota},\kappa)$, where
  $$
    \tK{G}{\iota}:=\{\,\{x,\iota x\}\mid x\in G\,\},
  $$
  and (writing $\-x$ for $\{x,\iota x\}$) $\kappa$ is defined by
  $$
  \begin{array}{*5c@{}l}
    \-x&\-y&\to&\-{x+y}&\-{x+\iota y}&.\\
  \end{array}
  $$

  For convenience, we define a \emph{twisted group} to be a 2-torsion
  abelian group together with an involution.
\end{definition}

\begin{remark}
  Again, it is easy to check the axioms, and that $0=\-0$ and
  $2\-x=\-{x+\iota x}$.

  This construction still works in the special case where $\iota$ is the
  identity, but since, for a 2-torsion group, $-1$ is also the identity,
  we get nothing new in this case: $\tK{G}{1}=\K{G}=G$.
\end{remark}

We shall show that every Kummer structure is either a Kummer
or a twisted Kummer, and describe the few which have both constructions.

\begin{example}
  We give one example of a structure which satisfies A1--A3 but not
  \ax{2ishom}, to show that \ax{2ishom} is not redundant.

  Let $Q_8=\{1,-1,i,-i,j,-j,k,-k\}$ be the quaternion group of
  order 8.  Quotienting $Q_8$ by the inverse operation gives a set of 5
  elements, $K=\{\-1,\-{-1},\-i,\-j,\-k\}$.
  It is straightforward to check that we can define
  $$
  \begin{array}{*5c@{}l}
    \-x&\-y&\to&\-{xy}&\-{xy^{-1}}&,\\
  \end{array}
  \quad 0=\-1,\quad 2\-x=\-{x^2},
  $$
  and the axioms A1--A3 then hold,
  but \ax{2ishom} fails: $\-i\ \-j\to\-k\ \-k$ should imply
  $\-{-1}\ \-{-1}\to \-{-1}\ \-{-1}$,
  but in fact $\-{-1}\ \-{-1}\to \-1\ \-1$.

  (We can regard $K$ as the `Kummer' of $Q_8$ -- but note that for most
  non-abelian groups we cannot even give a consistent
  definition for $\kappa$ of this form: to be consistent we need
  $\kappa\{\-x,\-y\}=\kappa\{\-y,\-x\}$ to hold.)
\end{example}

\section{First properties, and the 2-torsion group}\label{sec2torsion}

In this section $K$ shall be a Kummer structure.
We will collect some useful consequences of the axioms,
which lead us to a description of the \emph{2-torsion}
elements of $K$ -- \ie, elements in the kernel of the map~$2$.

\begin{lemma}
  The following hold for any $a,b,c,d\in K$.
  \begin{enumerate}[{L}1.]
  \def\itemlem#1{\item\label{lem#1}}%
  \itemlem{20=0} $20=0$.
  \itemlem{abc*} $a\ b\to c\ \star \iff b\ c\to a\ \star$.
  \itemlem{ab0*} $a\ b\to 0\ \star \iff a=b$.
  \itemlem{cd2a2b} $a\ b\to c\ d \implies c\ d\to 2a\ 2b$.
  \itemlem{abcc} Let $a\ b\to c\ d$.  Then $c=d \iff (2a=0$ or $2b=0)$.
  \itemlem{[2]} If $2a=2b=0$ then $a\ b\to c\ c$ for some $c$ with
    $2c=0$.
  \itemlem{pqrs} Given the notation of \ax{assoc}, there are also $r_0,r_1$
    such that
    $$
    \begin{array}{*5c@{}l}
      &c\\
      &a\\
      &\downarrow\\
      b&r_0&\to&s_0&s_3\\
      b&r_1&\to&s_2&s_1&.\\
    \end{array}
    $$
  \end{enumerate}
\end{lemma}

\begin{proof}
  \lem{20=0}.
    \ax{id}${} \implies 0\ 0\to 0\ 0$;
    but \ax{aa02a}${} \implies 0\ 0\to 0\ 20$;
    so $20=0$.

  \lem{abc*}.
    Given $a\ b\to c\ \star$, $b\ b\to 0\ 2b$ and
    $a\ 0\to a\ a$ are axioms, so we can fill in this case of \ax{assoc}:
    $$
    \begin{array}{*5c@{}l}
      &b&&b&b\\
      a&b&\to&c&\star\\
      &\downarrow&&\downarrow&\downarrow\\
      a&0&\to&a&a\\
      a&2b&\to&\star&\star&,\\
    \end{array}
    $$
    and read off $b\ c\to a\ \star$.
    The other implication follows by symmetry.

  \lem{ab0*}.
    Put $c=0$ in \lem{abc*}: by \ax{id}, $b\ 0\to a\ \star \iff a=b$.

  \lem{cd2a2b}.
    We first prove $c\ d\to 2a\ \star$:
    we have (using \ax{assoc} and \lem{abc*})
    $$
    \begin{array}{*5c@{}l}
      &c&&c&c\\
      a&b&\to&c&d\\
      &\downarrow&&\downarrow&\downarrow\\
      a&a&\to&p&q\\
      a&e&\to&r&s&,\\
    \end{array}
    $$
    and \ax{aa02a} implies $\{p,q\}=\{0,2a\}$.
    If $q=2a$ then we read off $c\ d\to 2a\ \star$.
    If $p=2a$ then we read off $c\ c\to 2a\ \star$,
    but also $c\ d\to 0\ \star$, so that (by \lem{ab0*}) $c=d$,
    so again $c\ d\to 2a\ \star$.

    Symmetry gives also $c\ d\to 2b\ \star$, so we're done unless $2a=2b$.

    By \ax{2ishom}, $2a\ 2b\to 2c\ 2d$, and if $2a=2b$ then \ax{aa02a} implies
    that one of $2c$ and $2d$ is $0$.  The two cases are the same, so take
    $2c=0$.  Now we can complete the diagram above: $p=r=0$ by \ax{aa02a};
    $e=a$ by \lem{ab0*}; $q=s=2a=2b$ by \ax{aa02a}; and we're done.

  \lem{abcc}.
    \lem{cd2a2b} says $a\ b\to c\ d \implies c\ d\to 2a\ 2b$; so, by
    \lem{ab0*}, $c=d \iff (2a=0$ or $2b=0)$.

  \lem{[2]}.
    \lem{abcc}${} \implies a\ b\to c\ c$ for some $c$;
    by \ax{2ishom}, $2a\ 2b\to 2c\ 2c$; but $\{2a,2b\}=\{0,0\}$ and
    $0\ 0\to 0\ 0$, so $2c=0$.

  \lem{pqrs}.
    Three applications of \ax{assoc} give
    $$
    \begin{array}{*5c@{}l}
      &c&&c&c\\
      a&b&\to&p_0&p_1\\
      &\downarrow&&\downarrow&\downarrow\\
      a&q_0&\to&s_0&s_1\\
      a&q_1&\to&s_2&s_3&,\\
    \end{array}\quad\:
    \begin{array}{*5c@{}l}
      &c&&c&c\\
      b&a&\to&p_0&p_1\\
      &\downarrow&&\downarrow&\downarrow\\
      b&r_0&\to&t_0&t_1\\
      b&r_1&\to&t_2&t_3&,\\
    \end{array}\quad\:
    \begin{array}{*5c@{}l}
      &b&&b&b\\
      a&c&\to&r_0&r_1\\
      &\downarrow&&\downarrow&\downarrow\\
      a&q_0&\to&u_0&u_1\\
      a&q_1&\to&u_2&u_3&.\\
    \end{array}
    $$
    Switching $r_0$ and $r_1$ if necessary, we may assume $t_0=s_0$ and
    $t_2=s_2$.  If $t_1=s_3$ and $t_3=s_1$, we have \lem{pqrs}, so assume
    $t_1=s_1\neq t_3=s_3$.  If $s_0=s_2$, we can switch $r_0$ and $r_1$ to give
    \lem{pqrs}, so assume also $s_0\neq s_2$.  We will now derive a
    contradiction.

    From the third diagram, $u_1\in\{s_0,s_1\}\cap\{t_2,t_3\}$, so either
    $u_1=s_0=s_3$ or $u_1=s_1=s_2$.  Switching $p_0$ and $p_1$ if necessary, we
    may assume $u_1=s_1=s_2$, and then $s_i=t_i=u_i$ for each $i$.

    Now (by \lem{cd2a2b})
    $$
    \begin{array}{*8c@{}l}
      a&q_0&\to&s_0&s_1&\to&2a&2q_0&,\\
      b&r_0&\to&s_0&s_1&\to&2b&2r_0&,\\
      c&p_0&\to&s_0&s_1&\to&2c&2p_0&,\\
    \end{array}
    $$
    so two out of $2a$, $2b$, and $2c$ are equal; without loss of generality
    $2a=2b$.  But
    $$
    \begin{array}{*8c@{}l}
      a&b&\to&q_0&q_1&\to&2a&2b&,
    \end{array}
    $$
    so by \lem{abcc}, $2q_0=0$ or $2q_1=0$.  By \lem{abcc} again, if $2q_0=0$
    then $s_0=s_1$, but we assumed $s_0\neq s_2=s_1$; if $2q_1=0$ then
    $s_2=s_3$, but we assumed $s_2=s_1\neq s_3$.  This is our contradiction.
\end{proof}

\begin{definition}
  The \emph{2-torsion} of $K$ is the set $K[2]:=\{\,a\in K\mid 2a=0\,\}$.
  By \lem{[2]}, if $a,b\in K[2]$ then $a\ b\to c\ c$ for some $c\in K[2]$;
  we define $a+b=c$ in this case.
\end{definition}

\begin{lemma}\label{lem2torsion}
  This construction makes $K[2]$ into a 2-torsion abelian group.
  For an abelian group $G$ we have $(\K{G})[2]\cong G[2]$, the 2-torsion in
  $G$.  For a twisted group $(G,\iota)$, we have
  $(\tK{G}{\iota})[2]\cong G[1+\iota]:=\ker(1+\iota)\subseteq G$.
\end{lemma}

\begin{proof}
  Clearly $+$ is commutative.
  For any $a,b,c\in K[2]$, \ax{assoc} gives us a diagram
  $$
  \begin{array}{*5c@{}l}
    &c&&c&c\\
    a&b&\to&(a+b)&(a+b)\\
    &\downarrow&&\downarrow&\downarrow\\
    a&(b+c)&\to&d&d\\
    a&(b+c)&\to&d&d&,\\
  \end{array}
  $$
  so $a+(b+c)=d=(a+b)+c$; that is, $+$ is associative.
  For any $a\in K[2]$ we have $a\ a\to 0\ 0$ by \ax{aa02a}, so $a+a=0$.
  Finally \ax{id} gives us $a+0=a$.
  So $K[2]$ is a 2-torsion abelian group.

  For $G$ an abelian group and $x,y\in G[2]$
  we have $\-x\ \-y\to\-{x+y}\ \-{x+y}$, so $\-x+\-y=\-{x+y}$;
  that is, $(x\mapsto\-x)$ is a group homomorphism $G[2]\to (\K{G})[2]$.
  It is clearly bijective, and so an isomorphism.

  Exactly the same argument applies to $(G,\iota)$ a twisted group and
  $x,y\in G[1+\iota]$.
\end{proof}

\begin{remark}
  In particular, this shows that 2-torsion Kummer structures are essentially
  the same as 2-torsion abelian groups (or as twisted groups with $\iota=1$).
\end{remark}

\section{Strings}\label{secStrings}

Since the previous section completely analyses 
2-torsion Kummer structures, from now on we assume that we have some
non-2-torsion element: in this section $K$ shall be a Kummer structure and
$g\in K\setminus K[2]$.

\begin{definition}
  Define $K_g\subseteq K^\ZZ$ to be the set of sequences
  $\alpha=(\alpha_n)_{n\in\ZZ}$ such that
  $\alpha_n\ g\to\alpha_{n-1}\ \alpha_{n+1}$ for all $n$.
  We call elements of $K_g$ \emph{strings}.
\end{definition}

\begin{lemma}\label{lemString}
  For any $n\in\ZZ$, any $\alpha_n,\alpha_{n+1}\in K$ such that
  $\alpha_n\ g\to \alpha_{n+1}\ \star$ can be extended to a unique
  $\alpha\in K_g$.
\end{lemma}

\begin{proof}
  \lem{abc*} says that the condition $\alpha_n\ g\to\alpha_{n+1}\ \star$ is
  symmetric in $\alpha_n$ and $\alpha_{n+1}$.

  Given such $\alpha_n$ and $\alpha_{n+1}$, there is a unique $\alpha_{n-1}$
  such that $\alpha_n\ g\to\alpha_{n+1}\ \alpha_{n-1}$ and a unique
  $\alpha_{n+2}$ such that $\alpha_{n+1}\ g\to\alpha_n\ \alpha_{n+2}$

  By induction on $n$ in both directions, this construction determines
  $\alpha_m$ for all $m\in\ZZ$.
\end{proof}

\begin{lemma}\label{lem2step}
  If $\alpha\in K_g$ then $\alpha_n\ 2g\to\alpha_{n-2}\ \alpha_{n+2}$
  for any $n$.
\end{lemma}

\begin{proof}
  By \ax{assoc}, \ax{aa02a}, and the definition of strings,
  $$
  \begin{array}{*5c@{}l}
    &g&&g&g\\
    \alpha_n&g&\to&\alpha_{n-1}&\alpha_{n+1}\\
    &\downarrow&&\downarrow&\downarrow\\
    \alpha_n&0&\to&\alpha_n&\alpha_n\\
    \alpha_n&2g&\to&\alpha_{n-2}&\alpha_{n+2}&.\\
  \end{array}
  $$
\end{proof}

\begin{definition}
  Define $\omicron$ to be the unique element of $K_g$ such that
  $\omicron_0=0$.
  (Its existence and uniqueness are guaranteed by Lemma~\ref{lemString}:
  since $0\ g\to g\ g$, the only possibility for $\omicron_1$ is $g$.)

  If $\alpha\in K_g$, define $\rho(\alpha)_n:=\alpha_{-n}$; clearly also
  $\rho(\alpha)\in K_g$.

  If $\alpha,\beta,\gamma,\delta\in K_g$ and for all $n,m\in\ZZ$,
  $$
  \begin{array}{*5c@{}l}
    \alpha_n&\beta_m&\to&\gamma_{n+m}&\delta_{n-m}&,\\
  \end{array}
  $$
  then we write $\gamma=\alpha+\beta$ and $\delta=\alpha-\beta$.

  If $\alpha,\beta,\~\gamma,\~\delta\in K_g$ and for all $n,m\in\ZZ$,
  $$
  \begin{array}{*5c@{\qquad}l}
    \alpha_n&\beta_m&\to&\~\gamma_{n+m}&\~\delta_{n-m}&
    \text{if }n\equiv m\pmod2,\\
    \alpha_n&\beta_m&\to&\~\gamma_{n-m}&\~\delta_{n+m}&
    \text{if }n\not\equiv m\pmod2,\\
  \end{array}
  $$
  then we write $\gamma=\alpha\oplus\beta$ and $\delta=\alpha\ominus\beta$.
\end{definition}

\begin{remark}
  In Lemma~\ref{lemPartialfunctions} of section~\ref{secGrid} we shall justify
  the functional notation by showing that these definitions make $+$, $-$,
  $\oplus$, and $\ominus$ into partial functions $K_g\times K_g\to K_g$.
  (\ie, for any $\alpha,\beta$ there is at most one $\gamma$ such that
  $\gamma=\alpha+\beta$, \etc..)

  For now, note that $\delta=\alpha-\beta\iff\delta=\alpha+\rho(\beta)$, and
  $\~\delta=\alpha\ominus\beta\iff\~\delta=\alpha\oplus\rho(\beta)$.
\end{remark}

\begin{theorem}\label{thmRecovery}
  Assume Lemma~\ref{lemPartialfunctions}, that $+$ and $\oplus$ are partial
  functions.
  Let $K$ be a Kummer structure and $g\in K\setminus K[2]$.

  $K\cong \K{G}$ for an abelian group $G$ if and only if taking $+$ as
  addition, $\omicron$ as zero, and $\rho$ as negation makes $K_g$ into a
  group isomorphic to $G$.

  $K\cong\tK{G}{\iota}$ for a twisted group $(G,\iota)$ if and only if taking
  $\oplus$ as addition, $\omicron$ as zero, and $\rho$ as the involution makes
  $K_g$ into a twisted group isomorphic to $(G,\iota)$.
\end{theorem}

\begin{proof}
  If $(+,\omicron,\rho)$ make $K_g$ into an abelian group,
  the map $K_g\to K$ taking $\alpha\mapsto\alpha_0$ clearly factors through a
  map $\K{K_g}=K_g/\rho\to K$; Lemma~\ref{lemString} implies that this map is
  a bijection $\K{K_g}\tofrom K$, and it is then straightforward to check that
  it is an isomorphism of Kummer structures.

  On the other hand, if $K=\K{G}$, we have $g=\-z=\{z,-z\}$ for some
  $z\in G\setminus G[2]$.
  Define a map $\phi\colon G\to K_g$ by $\phi(x)_n=\-{x+nz}$,
  which clearly makes $\phi(x)$ a $\-z$-string.
  Now $\phi$ is surjective, since for any $\alpha\in K_g$ with $\alpha_0=\-x$,
  either $\alpha_1=\-{x+z}$, in which case (by Lemma~\ref{lemString})
  $\alpha=\phi(x)$, or $\alpha_1=\-{x-z}$, in which case $\alpha=\phi(-x)$.
  And $\phi$ is injective, for if $\phi(y)=\phi(x)$ then $\-y=\-x$ and
  $\-{y+z}=\-{x+z}$, so $y=\pm x$ and $y+z=\pm(x+z)$; but these imply either
  $y=x$ or $2z=0$, and the latter we know is false.

  It is straightforward to check from the definitions above that
  $\phi(x+y)=\phi(x)+\phi(y)$ and $\phi(x-y)=\phi(x)-\phi(y)$, so that $+$ is
  an associative total function, and that $\phi(-x)=\rho(\phi(x))$ and
  $\phi(0)=0_{K_g}$, so that $K_g$ is an abelian group as claimed and
  $\phi\colon G\cong K_g$.

  The proof of the twisted group version of this result is similar, and we
  omit it.
\end{proof}

\begin{remark}
  The ideas of this section may also be used to define a natural action of the
  multiplicative monoid of $\NN$ on $K$: take $0a=0$, $1a=a$, and the rule
  $$
  \begin{array}{*5c@{}l}
    na&a&\to&(n-1)a&(n+1)a&\qquad(n>0).\\
  \end{array}
  $$
  Clearly $2a$ so defined is the same as the $2a$ we have been using, since
  \ax{aa02a} is a special case of the rule above.
  We do not make use of this construction, so we leave the interested reader
  to check that this is well-defined, and that this definition makes each
  $n\in\NN$ into a homomorphism in the sense of axiom \ax{2ishom}.
\end{remark}

\section{Colouring diamond grids}\label{secGrid}

In this section, $K$ shall be a Kummer structure, $g\in K\setminus K[2]$, and
$\alpha,\beta$ shall be two elements of $K_g$.  We wish to explore the question
of existence and uniqueness of $\alpha+\beta$, $\alpha-\beta$,
$\alpha\oplus\beta$, and $\alpha\ominus\beta$.

\begin{definition}
  Let $D$ be the graph whose nodes are pairs $(n,m)\in\ZZ^2$, with  $(n,m)$
  adjacent to $(n',m')$ if and only if $|n-n'|=|m-m'|=1$.  The edges of this
  graph lie in two directions: a \emph{rising} edge connects $(n,m)$ and
  $(n+1,m+1)$; a \emph{falling} edge connects $(n,m)$ and $(n+1,m-1)$.
  For each $p\in\{0,1\}$ let $D_p$ be the component of $D$ whose nodes are
  $(n,m)$ with $n+m\equiv p\pmod2$.

  If $\alpha_n\ \beta_m\to a\ b$, then call $a,b$ the \emph{values} at the node
  $(n,m)$.  When we draw parts of $D$ or $D_p$, we'll simply label each node
  with its unordered pair of values.
\end{definition}

\begin{lemma}\label{diamondrule}\emph{(The diamond rule)}
  Every diamond of nodes of $D_p$ looks like
  $$
  \begin{array}{*5{c@{\:}}l}
    &&c\ b\\
    &\diagup&&\diagdown\\
    a\ b&&&&c\ d&,\\
    &\diagdown&&\diagup\\
    &&a\ d\\
  \end{array}
  $$
  for some $a,b,c,d\in K$ (not necessarily distinct).
\end{lemma}

\begin{proof}
  In other words, for any $n,m\in\ZZ$ (with $n+m\not\equiv p\pmod2$),
  there are $a,b,c,d\in K$ such that
  $$
  \begin{array}{ccc@{}l}
    &\alpha_n\ \beta_{m+1}\to c\ b\\
    \alpha_{n-1}\ \beta_m\to a\ b&&\alpha_{n+1}\ \beta_m\to c\ d&.\\
    &\alpha_n\ \beta_{m-1}\to a\ d\\
  \end{array}
$$
But this is simply a case of \ax{assoc}:
$$
\begin{array}{*5c@{}l}
  &\alpha_n&&\alpha_n&\alpha_n\\
  \beta_m&g&\to&\beta_{m-1}&\beta_{m+1}\\
  &\downarrow&&\downarrow&\downarrow\\
  \beta_m&\alpha_{n-1}&\to&a&b\\
  \beta_m&\alpha_{n+1}&\to&d&c&.\\
\end{array}
$$
\end{proof}

\begin{lemma}\label{lemCentre}
  The node of $D_{1-p}$ in the middle of the diamond of Lemma~\ref{diamondrule}
  has values $u\ v$, where $u\ g\to a\ c$ and $v\ g\to b\ d$.
\end{lemma}

\begin{proof}
  \lem{pqrs}, applied to the instance of \ax{assoc} in the proof of
  Lemma~\ref{diamondrule}, states that $\alpha_n\ \beta_n\to u\ v$ such that
  $u\ g\to a\ c$ and $v\ g\to b\ d$.
\end{proof}

\begin{lemma}\label{linearrule}\emph{(The linear rule)}
  For any straight-line path in $D$, all the nodes in it have a value in
  common.
\end{lemma}

\begin{proof}
  From the diamond rule we see that any two adjacent nodes have a value in
  common, so in any exception to the lemma the nodes must contain a
  subsequence of the form
  $a\ c\relbar b\ c\relbar\dots\relbar b\ c\relbar b\ d$, where
  $a\neq b\neq c\neq d$, and $b\ c$ is repeated $n\geq 1$ times.
  Let us suppose we have such a path in $D$, with $n$ minimal.

  Pick one of the two adjacent parallel paths.
  Applying the diamond rule to the diamonds between the two paths, in turn
  starting from the left, the values on the neighbouring path are
  $a\ \star\relbar b\ \star\relbar \dots\relbar b\ \star\relbar b\ \star$.
  Now, starting from the right and applying the diamond rule again, we can
  fill in the remaining values: the neighbouring path also has values
  $a\ c\relbar b\ c\relbar \dots\relbar b\ c\relbar b\ d$.
  So, applying Lemma~\ref{lemCentre}, we have a part of $D$
  that looks (up to a possible $90^\circ$ rotation) like
  $$
  \def\lrlap#1#2{\text{\llap{$#1$}}\,\text{\rlap{$#2$}}}%
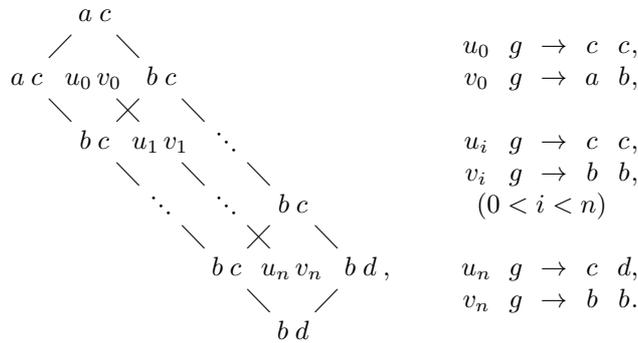
  \begin{array}{*{11}{c@{\:}}l}
    &&a\ c\\
    &\diagup&&\diagdown\\
    a\ c&&\lrlap{u_0}{v_0}&&b\ c\\
    &\diagdown&&\diagcross&&\diagdown\\
    &&b\ c&&\lrlap{u_1}{v_1}&&\diagdots&\\
    &&&\diagdown&&\diagdown&&\diagdown\\
    &&&&\diagdots&&\diagdots&&b\ c\\
    &&&&&\diagdown&&\diagcross&&\diagdown\\
    &&&&&&b\ c&&\lrlap{u_n}{v_n}&&b\ d&,\\
    &&&&&&&\diagdown&&\diagup\\
    &&&&&&&&b\ d\\
  \end{array}
  \qquad
  \begin{array}{*5c@{}l}
    u_0&g&\to&c&c&,\\
    v_0&g&\to&a&b&,\\
    \\
    u_i&g&\to&c&c&,\\
    v_i&g&\to&b&b&,\\
    \multicolumn{5}{c}{(0<i<n)}\\
    \\
    u_n&g&\to&c&d&,\\
    v_n&g&\to&b&b&.\\
  \end{array}
  $$

  Since the path from $u_0\ v_0$ to $u_n\ v_n$ is shorter than our minimal bad
  path, there must be a value in common between $\{u_0,v_0\},\ \dots,\
  \{u_n,v_n\}$.  Since $a\neq b\neq c\neq d$, the conditions on the right
  above make this impossible unless $n=1$, and then $v_0\neq v_1\neq u_0\neq
  u_1$, so the common value must be $v_0=u_1$, which implies $c=a$ and $d=b$.
  Renaming $u:=u_0$, $v:=v_0=u_1$, and $w:=v_1$, with one more application of
  the diamond rule we have
  $$
  \begin{array}{*7{c@{\:}}l}
    &&a\ a&&u\ \star\\
    &\diagup&&\diagcross&&\diagdown\\
    a\ a&&u\ v&&b\ a&&w\ \star\\
    &\diagdown&&\diagcross&&\diagcross\\
    &&b\ a&&w\ v&&b\ b&,\\
    &&&\diagdown&&\diagup\\
    &&&&b\ b\\
  \end{array}
  \qquad
  \begin{array}{*5c@{}l}
    u&g&\to&a&a&,\\
    v&g&\to&a&b&,\\
    w&g&\to&b&b&.\\
  \end{array}
  $$

  By Lemma~\ref{lemCentre}, either $b\ g\to u\ w$ or $a\ g\to u\ w$.  But by
  \lem{abc*}, since $v\ g\to a\ b$, both $a\ g\to v\ \star$ and
  $b\ g\to v\ \star$.  This is impossible, since we know $u\neq v\neq w$.
\end{proof}

\begin{definition}
  To \emph{colour} a node $N$ shall mean to assign values $\Gamma_N$ and
  $\Delta_N\in K$ such that $\{\Gamma_N,\Delta_N\}$ is the value set at $N$.
  A colouring of all the nodes of $D_p$ shall be said to be \emph{consistent}
  if whenever $AB$ is a falling edge, $\Gamma_A=\Gamma_B$, and whenever $AB$
  is a rising edge, $\Delta_A=\Delta_B$.
\end{definition}

\begin{lemma}\label{lemColour}
  $D_p$ has a consistent colouring, unique unless every node of $D_p$ has the
  same value set and that common value set is doubleton, in which case $D_p$
  has two consistent colourings.
\end{lemma}

\begin{proof}
  Say that an edge of $D_p$ is \emph{even} if the value sets of the two
  nodes it joins are equal.
  Looking at the diamond of Lemma~\ref{diamondrule}, the upper left edge
  is even if and only if $a=c$, if and only if the lower right edge is even.
  Similarly the upper right and lower left edges are each even if and only if
  $b=d$.  Thus even edges occur in infinite \emph{ladders} of parallel edges.

  Let $D'_p$ be the graph obtained by contracting all even edges of $D_p$.
  Because only even edges have been contracted, the nodes of $D'_p$ inherit
  well-defined value sets from the nodes of $D_p$.
  Because even edges form ladders, the underlying graph of $D'_p$ is
  isomorphic to a subgraph of $D_p$ consisting of the nodes
  $\{\,(n,m)\in D_p \mid \lambda_+\leq n+m\leq \mu_+
  \text{ and } \lambda_-\leq n-m\leq \mu_-\,\}$, for some
  $-\infty\leq\lambda_\pm\leq\mu_\pm\leq\infty$.
  Clearly the notion of rising and falling edges can be applied naturally to
  $D'_p$, and the diamond rule and the linear rule for $D'_p$ follow from the
  same for $D_p$.  For an even edge $AB$, whether rising or falling, the
  condition that a colouring be consistent is equivalent to
  $\Gamma_A=\Gamma_B$ and $\Delta_A=\Delta_B$, so consistent colourings of
  $D_p$ and of $D'_p$ are in one-to-one correspondence.

  $D'_p$ has no even edges, so given any edge $AB$ of $D'_p$, there is only
  one value in common between $A$ and $B$, so there is only one possible
  consistent colouring of just $A$ and $B$.
  (If the edge is falling, the common value must be $\Gamma_A=\Gamma_B$, if
  rising, $\Delta_A=\Delta_B$.)
  Say that the edge $AB$ \emph{forces} those particular colourings of $A$ and
  $B$.

  If there is at least one non-even edge in $D_p$, then every vertex $A$ of
  $D'_p$ meets an edge $AB$, which forces a particular colouring of $A$.
  So there is at most one consistent colouring of $D'_p$, and so of $D_p$.

  In this case, to prove the existence of a consistent colouring of $D'_p$, we
  must show that any two coincident edges $AB$, $BC$ force the same colouring
  on their common end $B$.

  But if $AB$ and $BC$ are not parallel, then they are part of a diamond,
  and the diamond rule states that any diamond can be consistently coloured.
  If $AB$ and $BC$ are parallel, they are a straight path, and the
  linear rule states that such a path can be consistently coloured.

  So $AB$ and $BC$ must force the same colouring on $B$.

  On the other hand, if every edge of $D_p$ is even, which is to say the value
  sets at all nodes of $D_p$ are equal, $D'_p$ is a single node, and any
  colouring of a single node is consistent; so there are two consistent
  colourings if the value set is doubleton and just one if the value set is
  singleton.
\end{proof}

\begin{lemma}\label{lemNot2x2}
  At most one of the components $D_0$ and $D_1$ has two distinct consistent
  colourings.
\end{lemma}

\begin{proof}
  Suppose we have a counterexample.  Then by Lemma~\ref{lemColour} all value
  sets at nodes of $D_p$ are equal, for each $p$.  So each diamond of $D_0$
  looks like
  $$
  \begin{array}{*5{c@{\:}}l}
    &&a\ b\\
    &\diagup&&\diagdown\\
    a\ b&&c\ d&&a\ b&,\\
    &\diagdown&&\diagup\\
    &&a\ b\\
  \end{array}
  \qquad
  \begin{array}{*5c@{}l}
    c&g&\to&a&a&,\\
    d&g&\to&b&b&.\\
  \end{array}
  $$
  Since $2g\neq 0$, by \lem{abcc} we must have $2c=2d=0$.  Similarly
  $2a=2b=0$.  But by \ax{2ishom}, if $c\ g\to a\ a$ then
  $2c\ 2g\to 2a\ 2a=0\ 0$.  So, by \lem{ab0*}, we have $2g=2c=0$, a
  contradiction.
\end{proof}

Now we can rephrase the definitions of $+$, $-$, $\oplus$, and $\ominus$ from
Section~\ref{secStrings} in terms of colourings of $D$:

\begin{lemma}\label{lemRedef+-}
  For $\gamma,\delta,\~\gamma,\~\delta\in K_g$, respectively
  $$
  \begin{array}{rcl}
    \gamma&=&\alpha+\beta;\\
    \delta&=&\alpha-\beta;\\
    \~\gamma&=&\alpha\oplus\beta;\\
    \~\delta&=&\alpha\ominus\beta\\
  \end{array}
  $$
  if and only if there is a consistent colouring
  $(\Gamma_{\cdot,\cdot},\Delta_{\cdot,\cdot})$ of $D$ such that, for all
  $n,m\in\ZZ$, respectively
  $$
  \begin{array}{rcl}
    \gamma_{n+m}&=&\Gamma_{n,m};\\
    \delta_{n-m}&=&\Delta_{n,m};\\
    \~\gamma_{n+m}&=&\left\{
    \begin{array}{l@{\qquad}l}
      \Gamma_{n,m}&\text{if }n\equiv m\pmod2,\\
      \Delta_{n,-m}&\text{if }n\not\equiv m\pmod2;\\
    \end{array}\right.\\
    \~\delta_{n-m}&=&\left\{
    \begin{array}{l@{\qquad}l}
      \Delta_{n,m}&\text{if }n\equiv m\pmod2,\\
      \Gamma_{n,-m}&\text{if }n\not\equiv m\pmod2.\\
    \end{array}\right.\\
  \end{array}
  $$
  \qed
\end{lemma}

\begin{lemma}\label{lemPartialfunctions}
  There can be at most one string in $K_g$ satisfying each of these conditions.
\end{lemma}

\begin{proof}
  This is clear from Lemma~\ref{lemRedef+-} in the case where $D$ has only one
  consistent colouring.  In the remaining case, by Lemma~\ref{lemColour} and
  Lemma~\ref{lemNot2x2}, there are two consistent colourings
  $(\Gamma_{\cdot,\cdot},\Delta_{\cdot,\cdot})$, 
  $(\Gamma'_{\cdot,\cdot},\Delta'_{\cdot,\cdot})$ of $D$: on one component
  of $D$, say $D_p$, the two are identical, while on the other, $D_{1-p}$, they
  satisfy $\Gamma'=\Delta$ and $\Delta'=\Gamma$.  Lemma~\ref{lemColour}
  further says that $\Gamma$ and $\Gamma'$ must each be constant on
  $D_{1-p}$, and not equal to each other.

  From $(\Gamma_{\cdot,\cdot},\Delta_{\cdot,\cdot})$ we can derive sequences
  $\gamma,\delta,\~\gamma,\~\delta\in K^\ZZ$ by the equations given in
  Lemma~\ref{lemRedef+-}, and in the same way from
  $(\Gamma'_{\cdot,\cdot},\Delta'_{\cdot,\cdot})$ we can derive
  $\gamma',\delta',\~\gamma',\~\delta'\in K^\ZZ$.

  Suppose both $\gamma$ and $\gamma'$ are in $K_g$.  Then
  $\gamma_p\ g\to\gamma_{p-1}\ \gamma_{p+1}$ and
  $\gamma'_p\ g\to\gamma'_{p-1}\ \gamma'_{p+1}$;
  in other words both $\Gamma_{p,0}\ g\to\Gamma_{p-1,0}\ \Gamma_{p+1,0}$ and
  $\Gamma'_{p,0}\ g\to\Gamma'_{p-1,0}\ \Gamma'_{p+1,0}$.
  But $\Gamma_{p,0}=\Gamma'_{p,0}$ whereas
  $\Gamma_{p-1,0}=\Gamma_{p+1,0}\neq\Gamma'_{p-1,0}=\Gamma'_{p+1,0}$, so
  this is impossible.

  Similarly, at most one of each of the pairs $\{\delta,\delta'\}$,
  $\{\~\gamma,\~\gamma'\}$, and $\{\~\delta,\~\delta'\}$ can be in $K_g$.
\end{proof}

\begin{remark}
  This fulfils the promise made in Section~\ref{secStrings}, to prove that
  $+$, $-$, $\oplus$, and $\ominus$ are partial functions.
\end{remark}

\section{Non-4-torsion Kummer structures}\label{secNon4torsion}

Recall that we defined $4g=2(2g)$.  In this section we consider the
constructions of the previous two sections in the case $4g\neq 0$,
and show that in this case $K$ is the Kummer of a group.
Throughout this section $r\equiv s$ always means modulo 2.
First we need a quick lemma:

\begin{lemma}\label{lemNota*a*a}
  If $\alpha\in K_g$ with $4g\neq0$ then it cannot be that
  $\alpha_{n-2}=\alpha_n=\alpha_{n+2}$ for any $n$.
\end{lemma}

\begin{proof}
  By Lemma~\ref{lem2step}, we would have $\alpha_n\ 2g\to\alpha_n\ \alpha_n$,
  so by \lem{abcc}, since $4g\neq0$ we must have $2\alpha_n=0$.
  But now applying \ax{2ishom} we have
  $2\alpha_{n+1}\ 2g\to2\alpha_n\ 2\alpha_{n+2}=0\ 0$, so by \lem{ab0*} we
  have $2\alpha_{n+1}=2g$ and by \lem{abcc}, $4g=0$ after all.
\end{proof}

\begin{theorem}\label{thmNon4torsion}
  If $K$ is a Kummer structure and there is any $g\in K$ with $4g\neq 0$ then 
  $K\cong \K{G}$ for an abelian group $G$.
\end{theorem}

\begin{proof}
  By Theorem~\ref{thmRecovery}, the theorem will follow if we can show
  that $+$ and $-$ are total functions $K_g\times K_g\to K_g$,
  and that $K_g$ is made into an abelian group by taking $+$ as addition,
  $\omicron$ as zero, and $\rho$ as negation: then $G\cong K_g$.

  First take $\alpha,\beta\in K_g$, construct the graph $D$ as in
  Section~\ref{secGrid} and choose a consistent colouring
  $(\Gamma_{\cdot,\cdot},\Delta_{\cdot,\cdot})$ of it.
  Since the colouring is consistent, we can define $\gamma,\delta\in K^\ZZ$ by
  (for all $n,m\in\ZZ$)
  $$
  \begin{array}{rcl}
    \gamma_{n+m}&=&\Gamma_{n,m};\\
    \delta_{n-m}&=&\Delta_{n,m}.\\
  \end{array}
  $$
  We need first to prove that $\gamma,\delta\in K_g$, so that we have
  $\gamma=\alpha+\beta$ and $\delta=\alpha-\beta$.

  For any $r\equiv s$, we can write $r=n+m$, $s=n-m$ and apply
  Lemma~\ref{lemCentre} to the diamond centred at $(n,m)$ to give
  \begin{align}
  \label{eqnstr}
  \mbox{either }&\left\{
  \begin{array}{*5c@{}l}
    \gamma_r&g&\to&\gamma_{r-1}&\gamma_{r+1},\\
    \delta_s&g&\to&\delta_{s-1}&\delta_{s+1};\\
  \end{array}
  \right.\\
  \label{eqnskew}
  \mbox{or }&\left\{
  \begin{array}{*5c@{}l}
    \gamma_r&g&\to&\delta_{s-1}&\delta_{s+1},\\
    \delta_s&g&\to&\gamma_{r-1}&\gamma_{r+1}&.\\
  \end{array}
  \right.
  \end{align}
  Now $\gamma,\delta\in K_g\iff{}$(\ref{eqnstr}) holds for all $r,s$.  So
  suppose there is some $r\equiv s\equiv p$ for which one of the
  equations of (\ref{eqnstr}) fails.  Then both equations of (\ref{eqnskew})
  hold, and so both equations of (\ref{eqnstr}) fail.  But one of these
  equations depends only on $r$ and the other only on $s$, so (\ref{eqnstr})
  must fail and (\ref{eqnskew}) hold for all $r\equiv s\equiv p$.

  Suppose (\ref{eqnstr}) fails for all $r,s\equiv p$ but holds for all
  $r,s\not\equiv p$.  Then for any $r\equiv p$, we have (by \lem{abc*})
  $$
  \begin{array}{*{11}c@{}l}
    \gamma_{r-1}&g&\to&\gamma_r&\star&\iff&\gamma_r&g&\to&\gamma_{r-1}&\star
      &,\\
    \gamma_{r+1}&g&\to&\gamma_r&\star&\iff&\gamma_r&g&\to&\gamma_{r+1}&\star
      &,\\
  \end{array}
  $$
  but $\gamma_r\ g\not\to\gamma_{r-1}\ \gamma_{r+1}$, so
  $\gamma_{r-1}=\gamma_{r+1}$ and (by \lem{abcc}) $2g=0$ or $2\delta_{r}=0$.
  The former is a contradiction, but given the latter for all
  $r\equiv p$, we can apply \ax{2ishom} to
  $\delta_{r+1}\ g\to\delta_r\ \delta_{r+2}$ to get $2\delta_{r+1}\ 2g\to0\
  0$,  whence $4g=0$, also a contradiction.

  Therefore if (\ref{eqnstr}) fails at all, it fails and (\ref{eqnskew})
  holds for all $r\equiv s$.  An application of \ax{assoc},
  $$
  \begin{array}{*5c@{}l}
    &g&&g&g\\
    \gamma_0&g&\to&\delta_{-1}&\delta_1\\
    &\downarrow&&\downarrow&\downarrow\\
    \gamma_0&0&\to&\gamma_0&\gamma_0\\
    \gamma_0&2g&\to&\gamma_2&\gamma_2&,\\
  \end{array}
  $$
  gives $\gamma_0\ 2g\to\gamma_2\ \gamma_2$, so by \lem{abcc}, either
  $4g=0$ (contradiction) or $2\gamma_0=0$.  Similarly $2\gamma_2=0$.
  But then \ax{2ishom} applied to $\delta_1\ g\to\gamma_0\ \gamma_2$ gives
  $2\delta_1\ 2g\to0\ 0$, so $4g=0$, contradiction again.

  So we always have $\gamma,\delta\in K_g$ and $+$ and $-$ are total
  functions.  We must show that $(+,\omicron,\rho)$ make $K_g$ into an abelian
  group; that is, that $+$ is associative and commutative, that
  $\alpha+\omicron=\alpha$ and that $\alpha+\rho(\alpha)=\omicron$.

  $+$ is obviously commutative.

  By definition
  $\alpha_n\ \omicron_0\to(\alpha+\omicron)_n\ (\alpha-\omicron)_n$.
  But $\omicron_0=0$, so $\alpha_n\ \omicron_0\to\alpha_n\ \alpha_n$.
  So $\alpha+\omicron=\alpha\;(\,=\alpha-\omicron\,)$.
  
  By definition
  $\alpha_n\ \rho(\alpha)_{-n}\to(\alpha+\rho(\alpha))_0\
  (\alpha-\rho(\alpha))_{2n}$.
  But $\rho(\alpha)_{-n}=\alpha_n$,
  so $\alpha_n\ \rho(\alpha)_{-n}\to 0\ \star$.
  So either $(\alpha+\rho(\alpha))_0=0$
  or $(\alpha-\rho(\alpha))_{2n}=0$.
  But the latter (for all $n$) contradicts Lemma~\ref{lemNota*a*a}, so
  $(\alpha+\rho(\alpha))_0=0$, which by definition of $\omicron$ means
  $\alpha+\rho(\alpha)=\omicron$.

  All that's left is to show $+$ to be associative.
  Fix $\alpha,\beta,\gamma\in K_g$, and set $\delta=\alpha+(\beta+\gamma)$
  and $\delta'=(\alpha+\beta)+\gamma$.

\begin{claim}
  For any $n$, if $\delta_n\neq\delta'_n$
  then either $\delta_{n+2}=\delta_n$ or $\delta_{n+4}=\delta_n$,
  and either $\delta_{n-2}=\delta_n$ or $\delta_{n-4}=\delta_n$.
\end{claim}

\begin{proof}
  For any $p+q+r=n$, consider this instance of \ax{assoc}:
  $$
  \begin{array}{*5c@{}l}
    &\gamma_r&&\gamma_r&\gamma_r\\
    \alpha_p&\beta_q&\to&(\alpha+\beta)_{p+q}&(\alpha-\beta)_{p-q}\\
    &\downarrow&&\downarrow&\downarrow\\
    \alpha_p&(\beta+\gamma)_{q+r}&\to&x&y\\
    \alpha_p&(\beta-\gamma)_{q-r}&\to&\star&\star&.\\
  \end{array}
  $$
  One of $x$ and $y$ is $(\alpha+(\beta+\gamma))_{p+q+r}=\delta_n$.
  Since $\delta_n\neq\delta'_n=((\alpha+\beta)+\gamma)_{p+q+r}$,
  we must have
  \begin{equation}\label{eqnEpsilon}
  \delta_n=\text{one of }\left\{
  \begin{array}{l}
    ((\alpha+\beta)-\gamma)_{p+q-r}=:(\epsilon^0)_{m_0},\\
    ((\alpha-\beta)+\gamma)_{p-q+r}=:(\epsilon^1)_{m_1},\\
    ((\alpha-\beta)-\gamma)_{p-q-r}=:(\epsilon^2)_{m_2}.\\
  \end{array}\right.
  \end{equation}
  The constraint $p+q+r=n$ is equivalent to
  $m_0\equiv m_1\equiv m_2\equiv n$ and $m_0+m_1-m_2=n$.

  For each $i$, we can't have $\delta_n=\epsilon^i_{m_i}$ for every
  $m_i\equiv n$, by Lemma~\ref{lemNota*a*a}.
  So for any permutation $\{i,j,k\}=\{0,1,2\}$, we can find $m_j$ and $m_k$
  such that $\epsilon^j_{m_j},\epsilon^k_{m_k}\neq\delta_n$.
  Choosing $m_i$ such that $m_0+m_1-m_2=n$, by (\ref{eqnEpsilon})
  we must have $\epsilon^i_{m_i}=\delta_n$.

  So each of the strings $\epsilon^i$ contains $\delta_n$, so (by
  Lemma~\ref{lemString}) each of these strings is obtained from $\delta$ by
  some translation and perhaps reversal.
  So if $\delta_{n+2},\delta_{n+4}\neq\delta_n$, or
  $\delta_{n-4},\delta_{n-2}\neq\delta_n$,
  we can find $m_0$ such that
  $\epsilon^0_{m_0},\epsilon^0_{m_0+2}\neq\delta_n$, and similarly $m_1$ such
  that $\epsilon^1_{m_1},\epsilon^1_{m_1+2}\neq\delta_n$.
  But then setting $m_2=m_0+m_1-n$, by (\ref{eqnEpsilon}) we must have
  $\epsilon^2_{m_2}=\epsilon^2_{m_2+2}=\epsilon^2_{m_2+4}=\delta_n$,
  contradicting Lemma~\ref{lemNota*a*a}.  This proves the Claim.
\end{proof}

  We can now complete the proof of Theorem~\ref{thmNon4torsion}.
  Suppose there is an $n$ such that $\delta_n\neq\delta'_n$.
  Then we can apply the Claim above.  Lemma~\ref{lemNota*a*a} excludes the case
  $\delta_{n-2}=\delta_n=\delta_{n+2}$, so we have either
  $\delta_n=\delta_{n+4}\neq\delta_{n+2}$ or
  $\delta_n=\delta_{n-4}\neq\delta_{n-2}$.
  We treat the former case; the latter is similar (and may be reduced to the
  former by applying $\rho$ to every string in question).
  
  By Lemma~\ref{lem2step}, we have $\delta_{n+2}\ 2g\to\delta_n\ \delta_n$ but
  $\delta'_{n+2}\ 2g\to\delta'_n\ \delta'_{n+4}$, so if
  $\delta'_n\neq\delta_n$ we can't have $\delta'_{n+2}=\delta_{n+2}$.

  But now the Claim above applies also to $\delta_{n+2}$, and since
  $\delta_n,\delta_{n+4}\neq\delta_{n+2}$, we must have
  $\delta_{n-2}=\delta_{n+2}=\delta_{n+6}$.

  Since $\delta_{n+2}\ 2g\to\delta_n\ \delta_n$ but $4g\neq0$, by
  \lem{abcc}, $2\delta_{n+2}=0$.
  Similarly, since $\delta_n\ 2g\to\delta_{n+2}\ \delta_{n+2}$, we have
  $2\delta_n=0$.

  But an application of \ax{2ishom} gives $2\delta_{n+1}\ 2g\to2\delta_n\
  2\delta_{n+2}=0\ 0$, so $4\delta_{n+1}=4g=0$, a contradiction.
  This completes the proof that $+$ is associative.
\end{proof}

\section{4-torsion Kummer structures}\label{sec4torsion}

In the previous section we have shown that every Kummer structure containing
an element $g$ with $4g\neq0$ is the Kummer of a group.
Here we turn to a different method to analyse the structure of a
\emph{4-torsion} Kummer structure -- \ie, one in which $4g=0$ for all elements
$g$.  So throughout this section $K$ shall be a 4-torsion Kummer structure.

To begin with, in section~\ref{sec2torsion} we showed that $K[2]$, the
2-torsion of $K$, is a 2-torsion abelian group.
For all $a\in K$, we have assumed $4a=0$, or equivalently $2a\in K[2]$.
So $2K$, the image of $2$, is contained in $K[2]$, and it is easy to show
(using \lem{20=0} and \ax{2ishom}) that $2K$ is a subgroup of $K[2]$.

If $a\in K$ and $x\in K[2]$, we can define $a+x$ by the rule
$a\ x\to(a+x)\ (a+x)$.

\begin{lemma}\label{lemAction}
  This extends the group operation in $K[2]$ to an action of $K[2]$ on $K$, and
  if $a\ b\to c\ d$ then $a\ (b+x)\to (c+x)\ (d+x)$.
  For any $a\in K$, $x\in K[2]$, we have $2(a+x)=2a$.
  If $a,b\in K$ with $2a=2b$ then we can write $b=a+x$ for $x\in K[2]$, and $x$
  is unique modulo $2a$.
\end{lemma}

\begin{proof}
  That $+$ is a group action means that $a+0=a$ and $a+(x+y)=(a+x)+y$.
  The former follows from \ax{id}, while the latter and the statement $a\ b\to
  c\ d \implies a\ (b+x)\to (c+x)\ (d+x)$ follow from \ax{assoc}:
  $$
  \begin{array}{*5c@{}l}
    &y&&y&y\\
    a&x&\to&(a+x)&(a+x)\\
    &\downarrow&&\downarrow&\downarrow\\
    a&(x+y)&\to&(a+x)+y&(a+x)+y\\
    a&(x+y)&\to&(a+x)+y&(a+x)+y&;\\
  \end{array}\quad\:
  \begin{array}{*5c@{}l}
    &x&&x&x\\
    a&b&\to&c&d\\
    &\downarrow&&\downarrow&\downarrow\\
    a&(b+x)&\to&(c+x)&(d+x)\\
    a&(b+x)&\to&(c+x)&(d+x)&.\\
  \end{array}
  $$
  By \ax{2ishom} we have $2a\ 0\to2(a+x)\ 2(a+x)$, so
  $2(a+x)=2a$.
  For $a,b$ with $2a=2b$, we know $2a\ 2b\to 0\ 0$, so if $a\ b\to x\ y$ then
  (by \ax{2ishom}) $2x=2y=0$, and (by \lem{abc*}) $a\ x\to b\ \star$, so
  $b=a+x$.
  Also $b=a+y$, but (by \lem{cd2a2b}) $x+y=2a$, so $x\equiv y\pmod{2a}$.
  Any $z$ with $b=a+z$ satisfies $a\ z\to b\ \star$, so $a\ b\to z\ \star$, so
  $z=x$ or $y$.
\end{proof}

Thus if for each $x\in 2K$ we pick a representative $e_x\in K$ with $2e_x=x$,
then every element of $a\in K$ can be written as $a=e_x+u$ with $x\in 2K$,
$u\in K[2]$; given $a$, $x=2a$ is unique and $u$ unique modulo $x$.

\begin{lemma}
  Given a choice of $\{e_x\}$, there are elements $\epsilon_{x,y}\in K[2]$
  for each $x,y\in 2K$, such that the structure map of $K$ is given by
  $$
  \begin{array}{*5c@{}l}
    (e_x+u)&(e_y+v)&\to&(e_{x+y}+u+v+\epsilon_{x,y})
    &(e_{x+y}+u+v+\epsilon_{x,y}+x)&.\\
  \end{array}
  $$
  $\epsilon_{x,y}$ is well-defined modulo $\langle x,y\rangle$.
\end{lemma}

\begin{proof}
  Let $e_x\ e_y\to a\ \star$.
  Then (by \ax{2ishom}) $2e_x\ 2e_y\to 2a\ \star$.  But
  $2e_x=x$ and $2e_y=y$, so $2a=x+y$.
  So $a=e_{x+y}+\epsilon$ for some $\epsilon\in K[2]$, defined (given the
  choice of $a$) modulo $(x+y)$.

  Now suppose $e_x\ e_y\to a\ b$.
  By \lem{cd2a2b}, $a\ b\to x\ y$, so $a\ x\to b\ \star$, so
  $b=a+x=e_{x,y}+\epsilon+x$.
  Thus choosing $b$ instead of $a$ changes $\epsilon$ by $x$, so
  overall the equation $e_x\ e_y\to (e_{x+y}+\epsilon_{x,y})\ \star$ defines
  $\epsilon_{x,y}$ modulo $\langle x,y\rangle$.

  Using this as the definition of $\epsilon_{x,y}$, two applications of
  Lemma~\ref{lemAction} give the full structure map of $K$ as in the lemma.
\end{proof}

A 4-torsion Kummer structure is thus determined up to isomorphism by the pair
of groups $2K\subseteq K[2]$ and the pairing
$(x,y)\in(2K)^2\mapsto\epsilon_{x,y}\in K[2]/\langle x,y\rangle$.
But the pairing $\epsilon_{\cdot,\cdot}$ is not canonical, since it
depends on the choice of representatives $\{e_x\}$.  We shall investigate its
properties and try to find a standard form into which we can put
$\epsilon_{\cdot,\cdot}$.  To begin with, we can take $e_0=0$, and from now
on we shall do so.  $\epsilon_{x,y}$ is only defined modulo $\langle
x,y\rangle$, so we take $\epsilon_{x,y}=z$ to mean the same as
$\epsilon_{x,y}\equiv z\pmod{\langle x,y\rangle}$.

\begin{lemma}\label{lemTetragon}
  For any $x,y,z\in 2K$,
  $$
    \epsilon_{x,y}+\epsilon_{x+y,z}\equiv\epsilon_{x,z}+\epsilon_{x+z,y}
    \pmod{\langle x,y,z\rangle}.
  $$
\end{lemma}

\begin{proof}
  Define $\delta\in K[2]$ by
  $$
  \begin{array}{*5c@{}l}
    &e_z&&e_z&e_z\\
    e_y&e_x&\to&(e_{x+y}+\epsilon_{x,y})&\star\\
    &\downarrow&&\downarrow&\downarrow\\
    e_y&(e_{x+z}+\epsilon_{x,z})&\to&(e_{x+y+z}+\delta)&\star\\
    e_y&\star&\to&\star&\star&.
  \end{array}
  $$
  The vertical arrow gives
  $\delta\equiv\epsilon_{x,y}+\epsilon_{x+y,z}\pmod{\langle x+y,z\rangle}$ and
  the horizontal arrow gives
  $\delta\equiv\epsilon_{x,z}+\epsilon_{x+z,y}\pmod{\langle x+z,y\rangle}$.
  Putting these together gives the lemma.
\end{proof}

\begin{lemma}\label{lemRank2}
  $\epsilon_{x,y}\in K[2]/\langle x,y\rangle$ depends only on the subgroup
  $\langle x,y\rangle\subseteq 2K$, not on the choice of generators $x,y$.
  If $\rank\langle x,y\rangle<2$ then $\epsilon_{x,y}=0$.
\end{lemma}

\begin{proof}
  \ax{aa02a}${} \implies e_x\ e_x\to 0\ x$.
  By definition, $e_x\ e_x\to (e_0+\epsilon_{x,x})\ \star$,
  so $\epsilon_{x,x}=e_0=0$.

  Now let $z=x+y$: Lemma~\ref{lemTetragon} becomes
  $\epsilon_{x,y}+\epsilon_{z,z}\equiv\epsilon_{x,z}+\epsilon_{y,y}
  \pmod{\langle x,y\rangle}$.  Since
  $\epsilon_{z,z}\equiv 0\equiv\epsilon_{y,y}\pmod{\langle x,y\rangle}$,
  we have $\epsilon_{x,y}\equiv\epsilon_{x,z}\pmod{\langle x,y\rangle}$;
  that is, $\epsilon_{x,y}=\epsilon_{x,z}$.
  This, together with the obvious symmetry
  $\epsilon_{x,y}=\epsilon_{y,x}$, shows that $\epsilon_{x,y}$ depends
  only on $\langle x,y\rangle$ (since these two symmetries generate the
  group $\GL_2(\ZZ/2)$).

  If $\rank\langle x,y\rangle<2$ then $\langle x,y\rangle$ is generated by a
  single element $z$, and $\epsilon_{x,y}=\epsilon_{z,z}=0$.
\end{proof}

In the light of the last lemma, if $V=\langle x,y\rangle$ then we shall use
the notation $\epsilon_V=\epsilon_{x,y}$.

\begin{lemma}\label{lemStandardform}
  Let $\{x_1,\dots,x_n\}$ be a basis for a rank-$n$ subgroup $X\subseteq 2K$.
  Call two elements of this subgroup \emph{disjoint} if they are sums of
  disjoint subsets of the basis.
  Then we can choose $\{e_x\}$ such that when $a,b\in X\setminus 0$ are
  disjoint, $\epsilon_{a,b}=0$, and if $a,b,c\in X\setminus0$ are disjoint,
  $\epsilon_{a+b,a+c}\equiv0\pmod{\langle a,b,c\rangle}$.
\end{lemma}

\begin{proof}
  We work by induction on $n$: for $n=0$ there is nothing to prove, so
  let $X'=\langle x_1,\dots,x_{n-1}\rangle$ and suppose we have
  already chosen $e_x$ for all $x\in X'$ satisfying the given conditions on
  $\epsilon_V$ for $V\subseteq X'$.  Define $T=x_1+\dots+x_n$ and $T'=T-x_n$.

  Choose $e_{x_n}$ arbitrarily (subject to $2e_{x_n}=x_n$).
  Then for each $a\in X'\setminus 0$, we can pick $e_{a+x_n}$ such that
  $e_a\ e_{x_n}\to e_{a+x_n}\ \star$, so that $\epsilon_{a,x_n}=0$.
  There are two such choices, differing by $x_n$.
  If also $a\neq T'$, then by Lemma~\ref{lemTetragon},
  $$
    \epsilon_{a,x_n}+\epsilon_{a+x_n,T'-a}\equiv
    \epsilon_{a,T'-a}+\epsilon_{T',x_n}
    \pmod{\langle a,T'-a,x_n\rangle}.
  $$
  By induction $\epsilon_{a,T'-a}=0$ and by the choices just made
  $\epsilon_{a,x_n}=0$ and $\epsilon_{T',x_n}=0$.
  So $\epsilon_{a+x_n,T'-a}\equiv0\pmod{\langle a+x_n,T'-a,x_n}$.
  By adding $x_n$ to $e_{a+x_n}$ if necessary, we can make
  $\epsilon_{a+x_n,T'-a}=0$.
  So we have now $\epsilon_{a,b}=0$ whenever $a\in X'$ and $b=x_n$ or
  when $a+b=T$.

  Now let $a,b\in X$ be disjoint.  If $a,b\in X'$, we know $\epsilon_{a,b}=0$
  by induction.
  So assume $a=a'+x_n$ with $a',b\in X'$. 
  Lemma~\ref{lemTetragon} says
  $$
    \epsilon_{a',x_n}+\epsilon_{a,b}\equiv
    \epsilon_{a',b}+\epsilon_{a'+b,x_n}
    \pmod{\langle a,b,x_n\rangle}.
  $$
  Three of these terms we know to be $0$,
  leaving $\epsilon_{a,b}\equiv0\pmod{\langle a,b,x_n\rangle}$.
  Also by Lemma~\ref{lemTetragon},
  $$
    \epsilon_{b,a}+\epsilon_{b+a,T-a-b}\equiv
    \epsilon_{b,T-a-b}+\epsilon_{T-a,a}
    \pmod{\langle a,b,T\rangle}.
  $$
  Again three of these terms we know to be $0$, leaving
  $\epsilon_{b,a}=\epsilon_{a,b}\equiv0\pmod{\langle a,b,T\rangle}$.
  If $T\notin \langle a,b,x_n\rangle$, we can put these together to get
  $\epsilon_{a,b}=0$.
  But if $T\in\langle a,b,x_n\rangle=\langle a',b,x_n\rangle$
  then $T=a'+b+x_n=a+b$,
  since $a'$, $b$, and $x_n$ are disjoint, and in this case $\epsilon_{a,b}=0$
  from the choices made in the previous paragraph.

  Finally let $a,b,c\in X$ be disjoint and non-zero.
  Once more Lemma~\ref{lemTetragon} says
  $$
    \epsilon_{a,b}+\epsilon_{a+b,a+c}\equiv
    \epsilon_{a,a+c}+\epsilon_{c,b}
    \pmod{\langle a,b,c\rangle},
  $$
  and $\epsilon_{a,b}=0$, $\epsilon_{a,a+c}=\epsilon_{a,c}=0$,
  $\epsilon_{c,b}=0$, leaving
  $\epsilon_{a+b,a+c}\equiv0\pmod{\langle a,b,c\rangle}$ as required.
\end{proof}

\begin{definition}
  Let $X=\langle x,y,z\rangle\subseteq 2K$ be a rank-3 subgroup.
  Define $\ind_X=0$ if there is a choice of $\{e_x\}$ such that
  $\epsilon_V=0$ for all rank-2 $V\subset X$;
  otherwise, define $\ind_X=1$.
\end{definition}

\begin{lemma}\label{lemIndX}
  Let $X=\langle a,b,c\rangle$ be rank 3,
  and let $\{e_x\}$ be chosen as allowed by Lemma~\ref{lemStandardform}
  so that $\epsilon_V=0$ for all $V\subset X$
  except possibly for $V=U:=\langle a+b,a+c\rangle$.
  Then $\epsilon_U=0\iff\ind_X=0$.
\end{lemma}

\begin{proof}
  We know $\epsilon_U\equiv0\pmod X$.
  If $\epsilon_U=0$ then we have $\ind_X=0$ by definition.  So all we need to
  prove is that if $\ind_X=0$ then $\epsilon_U=0$.

  Suppose $\ind_X=0$.  Then there is another choice of representatives
  $\{e'_x\}$ giving $\epsilon'_V=0$ for all $V\subset X$.
  By Lemma~\ref{lemAction}, each $e'_x=e_x+\delta_x$ for some
  $\delta_x\in K[2]$.  So $e'_x\ e'_y\to
  (e'_{x+y}+\epsilon_{x,y}+\delta_x+\delta_y+\delta_{x+y})\ \star$; but
  for all $x,y\in X$, we know $e'_x\ e'_y\to e'_{x+y}\ \star$, so
  $$
    \epsilon_{x,y}\equiv\delta_x+\delta_y+\delta_{x+y}
    \pmod{\langle x,y\rangle}.
  $$
  Set $\Delta:=\delta_a+\delta_b+\delta_c+\delta_{a+b+c}$.
  Since $\epsilon_V=0$ for all $U\neq V\subset X$, we have
  \begin{align*}
    \delta_{a+b}+\delta_a+\delta_b&\equiv0\pmod{\langle a,b\rangle}\\
    \delta_{a+b}+\delta_c+\delta_{a+b+c}&\equiv0\pmod{\langle a+b,c\rangle},
  \end{align*}
  and, adding these together, we can see that $\Delta\in X$ and
  \begin{align*}
    \Delta\in\langle a+b,c\rangle&\iff
    \delta_{a+b}+\delta_a+\delta_b\equiv0\pmod{\langle a+b\rangle}\\
    &\iff\delta_{a+b}+\delta_a+\delta_b\in U.
  \end{align*}
  Analogous statements hold for other permutations of $\{a,b,c\}$.

  Now it is easy to check that every element of $X$ is in an even number of
  the subgroups $\langle a+b,c\rangle$, $\langle a+c,b\rangle$, and $\langle
  b+c,a\rangle$.  So an even number out of $(\delta_{a+b}+\delta_a+\delta_b)$,
  $(\delta_{a+c}+\delta_a+\delta_c)$, and $(\delta_{b+c}+\delta_b+\delta_c)$
  are in $U$, so their sum is in $U$.  But their sum is
  $$
    \delta_{a+b}+\delta_{a+c}+\delta_{b+c}\equiv\epsilon_U\pmod U,
  $$
  so $\epsilon_U=0$, as required.
\end{proof}

\begin{lemma}\label{lemIndK}
  If $X$ and $Y$ are two rank-3 subgroups of $2K$ then $\ind_X=\ind_Y$.
\end{lemma}

\begin{proof}
  It is enough to show this when $\rank(X\cap Y)=2$, since any two rank-3
  subgroups can be connected by a chain of subgroups with each 2 consecutive
  members of the chain meeting having rank-2 intersection.

  So we may assume $X=\langle a,b,c\rangle$ and $Y=\langle b,c,d\rangle$,
  and choose $\{e_x\}$ by Lemma~\ref{lemStandardform}.
  Let $U:=\langle a+b,a+c,a+d\rangle$, the group of even sums of
  $\{a,b,c,d\}$.  By Lemma~\ref{lemTetragon},
  $$
    \epsilon_{a+b,a+c}+\epsilon_{b+c,c+d}\equiv
    \epsilon_{a+b,c+d}+\epsilon_{a+b+c+d,b+c}
    \pmod{U}.
  $$
  But according to Lemma~\ref{lemStandardform}, the two terms on the
  right-hand side of this are 0, while by Lemma~\ref{lemIndX},
  $\epsilon_{a+b,a+c}\in U\iff\ind_X=0$ and
  $\epsilon_{b+c,c+d}\in U\iff\ind_Y=0$.
\end{proof}

So if $\rank 2K\geq3$ we can define $\ind_K:=\ind_X$ for any rank-3
$X\subseteq K$.
Finally we can classify 4-torsion Kummer structures:

\begin{theorem}\label{thm4torsion}
  Let $K$ be a 4-torsion, but not 2-torsion, Kummer structure.

  $K\cong \K{G}$ for an abelian group $G$ if and only if $\rank 2K\leq2$ or
  $\ind_K=1$.

  $K\cong \tK{G}{\iota}$ for a twisted group $(G,\iota)$ if and only if
  $\rank 2K\leq2$ or $\ind_K=0$.
\end{theorem}

\begin{proof}
  For the only-if parts, note that for $K\cong \K{G}$, a rank-3 $X\subseteq
  2K$ is the image of a rank-3 subgroup of $2G$, which is $2\times$ a
  subgroup of $G$ isomorphic to $(\ZZ/4)^3$.  It is easy to check then
  that $\ind_X=1$.
  Similarly, if $K\cong \tK{G}{\iota}$, a rank-3 $X\subseteq 2K$ is the image
  of a rank-3 subgroup of $(1+\iota)G$, which is $(1+\iota)$ times a subgroup
  isomorphic to $((\ZZ/2)^2,\iota_2)^3$, where $\iota_2$ exchanges the two
  factors of $(\ZZ/2)^2$.  It is again easy to check that $\ind_X=0$.

  For the if parts, we have shown that $K$ is determined up to isomorphism
  by the two groups $2K\subseteq K[2]$ and a pairing
  $(x,y)\in(2K)^2\mapsto\epsilon_{x,y}\in K[2]/\langle x,y\rangle$.
  If $K$ is finite, for any basis of $2K$ we can put $\epsilon_{\cdot,\cdot}$
  in standard form as in Lemma~\ref{lemStandardform};
  then by Lemmas \ref{lemIndX} and~\ref{lemIndK},
  the pairing $\epsilon_{\cdot,\cdot}$ is entirely determined by $\ind_K$.
  So a finite 4-torsion Kummer structure $K$ is determined up to isomorphism
  by $\rank 2K$, $\rank K[2]$, and, if $\rank 2K\geq3$, $\ind_K$.

  But if $G=(\ZZ/4)^a\oplus(\ZZ/2)^b$ and $K=\K{G}$ then $\rank 2K=a$,
  $\rank K[2]=a+b$, and, if $a\geq3$, $\ind_K=1$.
  And if $(G,\iota)=((\ZZ/2)^2,\iota_2)^a\oplus(\ZZ/2,1)^b$ and
  $K=\tK{G}{\iota}$ then $\rank 2K=a$, $\rank K[2]=a+b$, and, if $a\geq3$,
  $\ind_K=0$.

  This proves the theorem for finite $K$.  For infinite $K$, we can see from
  Theorem~\ref{thmRecovery} that the property of being isomorphic to the
  Kummer of a group, and the property of being isomorphic to a twisted Kummer,
  are each equivalent to a collection of statements that depend only on the
  structure of finitely generated substructures (for example, the statement
  that $\alpha+(\beta+\gamma)=(\alpha+\beta)+\gamma$ in $K_g$ depends only on
  the substructure generated by $\{\alpha_0,\beta_0,\gamma_0,g\}$).
  A finitely generated substructure $L$ of a 4-torsion Kummer structure $K$
  is finite, and if one of the properties $\rank 2K\leq 2$ or
  $\ind_K=i$ holds, the same property holds for every
  sub-Kummer-structure of $K$, so the theorem can be extended from
  finite $K$ to arbitrary $K$.
\end{proof}

\section{Summary}

Collecting together Theorems \ref{thmRecovery}, \ref{thmNon4torsion}, and
\ref{thm4torsion}, and the Remark after Lemma~\ref{lem2torsion}, we have the
following classification of Kummer structures:

\begin{theorem}
  Every Kummer structure is either the Kummer of an abelian group or a
  twisted Kummer.  If $K\cong \K{G}$ then $G$ is determined up to isomorphism
  by $K$.  If $K\cong \tK{G}{\iota}$ then $(G,\iota)$ is determined up to
  isomorphism by $K$.  The only Kummers that are also twisted Kummers are
  $\K{G}$ where either $G$ is 2-torsion or $G$ is 4-torsion and $\rank
  2G\leq2$, which are isomorphic to $\tK{G}{\iota}$ where $\rank
  (1+\iota)G\leq2$.
  \qed
\end{theorem}

We can therefore characterise Kummers of groups by adding to
\axioms/ the statement that $K$ has no substructure isomorphic to the unique
twisted Kummer $L$ with $\rank 2L=\rank L[2]=3$, which is
$\tK{(\ZZ/2)^6}{\iota}$, where the involution $\iota$ acts on a basis
$\{x_0,\dots,x_5\}$ by $x_i\mapsto x_{(i+3)\bmod6}$.

\begin{remark}
  In fact, $G$ or $(G,\iota)$ is determined by $K$ in a constructive and
  computable way.
  For 2-torsion $K$ this is trivial, while for non-2-torsion $K$
  Theorem~\ref{thmRecovery} shows that the elements of $G$ can be represented
  as strings, and Lemma~\ref{lemString} shows that a string can be represented
  by a pair of elements of $K$.  The string $\gamma=\alpha+\beta$ or
  $\gamma=\alpha\oplus\beta$ of Section~\ref{secStrings} is determined up to
  finite ambiguity from $\alpha_0\ \beta_0\to\gamma_0\ \star$, and the proof
  of Lemma~\ref{lemPartialfunctions} shows that this ambiguity can be removed
  by extending the strings to a sufficient (finite) length and considering all
  the conditions on $\gamma_0$ and $\gamma_1$ that then arise.
\end{remark}

\end{document}